\documentclass[12pt]{article}
\usepackage[margin=2cm]{geometry}\usepackage{CJK}
\usepackage{amsfonts}
\usepackage{amssymb}
\usepackage{amsthm}
\usepackage{amsmath}
\usepackage{latexsym}
\usepackage{color}
\usepackage{mathrsfs}
\begin{document}
  \begin{CJK*}{GBK}{kai}
\newtheorem{lemma}{Lemma}
\newtheorem{pron}{Proposition}
\newtheorem{rem}{Remark}
\newtheorem{thm}{Theorem}
\newtheorem{Corol}{Corollary}
\newtheorem{exam}{Example}
\newtheorem{defin}{Definition}
\newtheorem{remark}{Remark}
\newtheorem{property}{Property}
\newtheorem{assh}{Assumption}
\newtheorem{assd}{Assumption}
\newcommand{\rco}{\color{red}}
\newcommand{\la}{\frac{1}{\lambda}}
\newcommand{\sectemul}{\arabic{section}}
\renewcommand{\theequation}{\sectemul.\arabic{equation}}
\renewcommand{\thepron}{\sectemul.\arabic{pron}}
\renewcommand{\thelemma}{\sectemul.\arabic{lemma}}
\renewcommand{\therem}{\sectemul.\arabic{rem}}
\renewcommand{\thethm}{\sectemul.\arabic{thm}}
\renewcommand{\theCorol}{\sectemul.\arabic{Corol}}
\renewcommand{\theexam}{\sectemul.\arabic{exam}}
\renewcommand{\thedefin}{\sectemul.\arabic{defin}}
\renewcommand{\theremark}{\sectemul.\arabic{remark}}
\renewcommand{\theassh}{H\arabic{assh}}
\renewcommand{\theassd}{D\arabic{assd}}
\def\REF#1{\par\hangindent\parindent\indent\llap{#1\enspace}\ignorespaces}
\def\lo{\left}
\def\ro{\right}
\def\be{\begin{equation}}
\def\ee{\end{equation}}
\def\beq{\begin{eqnarray*}}
\def\eeq{\end{eqnarray*}}
\def\bea{\begin{eqnarray}}
\def\eea{\end{eqnarray}}
\def\d{\Delta_T}
\def\r{random walk}
\def\o{\overline}

\title{\large\bf The finite-time ruin probability of the nonhomogeneous Poisson risk model with conditionally independent subexponential claims\thanks{Research supported by National Natural Science Foundation of China
(No.s 11401415), and the Priority Academic Program Development of Jiangsu Higher Education
Institutions.}}
\author{\small
Hui Xu,~~ Fengyang Cheng
\thanks{Corresponding author.
Telephone: +86 512 65112637. Fax: +86 512 65112637. E-mail:
chengfy@suda.edu.cn}
\\
{\footnotesize\it School of Mathematical Sciences, Soochow
University, Suzhou 215006, China}}
\date{}

\maketitle

\begin{center}
{\noindent\small {\bf Abstract }}
\end{center}

{\small This paper obtains an asymptotic formula for the finite-time ruin probability of the compound nonhomogeneous Poisson risk model with a constant interest force, in which the claims are conditionally independent random variables with a common subexponential distribution. The paper also obtains some asymptotic relations of randomly weighted sums $\sum_{i=1}^n \theta_iX_i$, in which the weights $\theta_i$ $i=1,2,\cdots, n$ are nonnegative random variables which are bounded above and the primary random variables $X_i$, $i=1,2,\cdots,n$ are conditionally independent and follow subexponential distributions.
\medskip

\medskip

}

\section{Introduction}
\setcounter{thm}{0}\setcounter{Corol}{0}\setcounter{lemma}{0}\setcounter{pron}{0}\setcounter{equation}{0}
\setcounter{remark}{0}\setcounter{exam}{0}\setcounter{property}{0}\setcounter{defin}{0}\

In this paper, we will focus on the compound nonhomogeneous Poisson risk model with a constant interest force, in which the claim sizes $X_k, k=1,2,\cdots,$ form a sequence of conditionally independent nonnegative random variables (r.v.s) with a common distribution $F$, while the arrival times $\sigma_k, k=1,2,\cdots$, constitute a nonhomogeneous Poisson process $$ N(t)=\sup \{k:~\sigma_k\leq t\},~ t\geq 0$$ with an intensity function $\lambda(t)$ and $\sigma_0=0$ by convention. Let $\{C(t): t\geq 0\}$ be a nondecreasing and right-continuous stochastic process that represents the total premium accumulated up to time $t$ and let $r>0$ be the constant interest force (i.e. one dollar becomes $e^{rt}$ dollars after time $t$).  Then the total surplus up to time $t$, denoted by $S_{r}(t)$, can be expressed as
\begin{eqnarray}\label{ruin0}
S_r(t)=xe^{rt}+\int_{0}^{t}e^{r(t-s)}C(ds)-\sum_{k=1}^{N(t)}X_ke^{r(t-\sigma_k)},\ t\ge0,
\end{eqnarray}
where $\sum_{k=1}^0X_k=0$ by convention, and $x\ge0$ is the initial surplus of an insurance company.

Correspondingly, the ruin probability with a finite time $T>0$ can be defined as
\begin{eqnarray}\label{ruin1}
\psi_r(x,T)=P(S_r(t)<0 \text{~for some}\ 0\le t\le T|S_r(0)=x).
\end{eqnarray}

When $\lambda(t)\equiv\lambda>0$ (i.e. $N(t)$ is a homogeneous Poisson process) and the claims $X_i,i=1,2,\cdots$ are independent, Tang (2005) proved that if the claims follow a common subexponential distribution $F$ (the definition is arranged in section 2.1), then for each $T>0$, it holds that
\begin{eqnarray}\label{00}
\psi_r(x,T)\sim\frac{\lambda}{r}\int_{x}^{xe^{rT}}\frac{\overline F(y)}{y}dy
\end{eqnarray}
as $x\to\infty$.

Many scholars have been attracted by this interesting result and they have tried to
generalize the result in several directions.  It is well known that, if $\{N(t): t\geq 0\}$ is a homogeneous Poisson process, then its inter arrivals $\theta_k=\sigma_k-\sigma_{k-1}$, $k=1,2,\cdots$
 form a sequence of independent and identically distributed r.v.s with a common exponential distribution. Hence, the first direction is to assume that $\{N(t):t\geq 0\}$ is a renewal process or a generalized renewal process, in which the inter arrival times $\theta_k$, $k=1,2,\cdots$  have a common distribution, not necessarily be exponential,  see Chen and Ng (2007), Wang (2008), Yang and Wang (2010) etc. Another direction is to assume that the claim sizes $X_k, k=1,2,\cdots,$ are dependent r.v.s with a common distribution,
 see Kong and Zong (2008), Yang and Wang (2012) and Gao et al(2012) etc.

In practical applications, the assumption that the inter arrival times $\theta_k $, $k=1,2,\cdots$ have a common distribution is unrealistic. For example, summer floods often lead to a lot of personal and property damage, which would cause that the inter arrival times are significantly different
in summer and winter. In general, for many insurance projects, the greater the number of insured persons, the greater the total number of claims,
which would lead to the assumption that the inter arrival times follows a common distribution unreasonable.
 Therefore, it is necessary to discuss the risk model with differently distributed inter arrival times.

On the other hand, more and more attention is paid to situations that the claim sizes $X_k, k=1,2,\cdots,$ are dependent r.v.s with a common distribution $F$. To the best of our knowledge, in this case, it is always assumed that $F$ belongs to the intersection of long-tailed class and dominant variation class (their definitions can be found in section 2.1) or its subclass, which excluded many popular subexponential distributions such as the lognormal distribution, the Weibull distributions with parameters in $(0,1)$, etc. and hence, these results can not cover Theorem 3.1 in Tang (2005). For this purpose, a question naturally arises: under what conditions, relation (\ref{00}) still holds for claim sizes are dependent r.v.s with a common subexponential distribution $F$?

In this paper, we will extend Theorem 3.1 in Tang (2005) in two directions: Firstly, we assume that $\{N(t):t\ge0\}$ is a nonhomogeneous Poisson process, which result in that the inter arrival times are neither independent nor identically distributed; Secondly, we assume that the claim sizes are conditionally independent.

We  specifically point out that the assumption of that the number of claims is a nonhomogeneous Poisson process is very reasonable. In fact, assuming that the claim of each insured of an insurance project is independent of each other,
 then the number of claims in any time interval will follows a binomial distribution, which follows approximately a Poisson distribution by the Poisson Theorem since the probability of an insured claim is usually small and the number of insured is usually large.

The rest of this paper consists of two sections. In Section 2, we will introduce some heavy-tailed distribution classes and some assumptions among random variables, and present main results of this paper. The proofs of Theorems are arranged in Section 3.

\section{Preliminaries and main results}
\setcounter{thm}{0}\setcounter{Corol}{0}\setcounter{lemma}{0}\setcounter{pron}{0}\setcounter{equation}{0}
\setcounter{remark}{0}\setcounter{exam}{0}\setcounter{property}{0}\setcounter{defin}{0}\

Throughout this paper, all limit relationships are for $x\to\infty$ unless stated otherwise, and  all random variables are defined on a probability space $(\Omega,\mathscr{F},\textbf{P})$. For two
positive functions $a(x)$ and $b(x)$, we write $a(x)\sim b(x)$ if $\lim a(x)/b(x)=1$; write $a(x)=o(b(x))$ if $\lim a(x)/b(x)=0$; write $a(x)=O(b(x))$ if $\limsup a(x)/b(x)<\infty$ and write $a(x)\lesssim b(x)$ if $\limsup a(x)/b(x)\leq 1$. For two real numbers $x$ and $y$, denote $x\vee y=\max(x,y)$ and $x\wedge y=\min(x,y)$. For any distribution $F$, we denote its (right) tail by $\overline F(x)=1-F(x)=F(x,\infty)$, $x\in(-\infty,\infty)$.

\subsection{Some distribution classes}\

To model the dangerous claim sizes in the insurance industry, most practitioners select the claim-size distribution from the heavy-tailed distribution
class. By definition, a distribution $F$ is said to be heavy-tailed if $\int_{0}^{\infty}e^{\varepsilon y}F(dy)=\infty$ holds for all $\varepsilon>0$.
 To this end, we now introduce some important subclasses of heavy-tailed
distribution class, one of which is the subexponential distribution class.

 A distribution $F$ supported on $[0,+\infty)$ is said to be subexponential, denoted by $F\in\mathcal{S}$, if $F$ is unbounded above ( i.e. $\overline{F}(x)>0$ holds for all $x>0$) and the relation
\begin{eqnarray}
\overline {F^{*n}}(x)\sim n\overline F(x)
\end{eqnarray}
holds for some (or equivalently for all) $n=2,3,\cdots$,  where $F^{*n}$ denotes the n-fold convolution of $F$ with itself. Furthermore, a distribution $F$ supported on $(-\infty,+\infty)$ is still called subexponential, if $F^+$ is subexponential, where $F^+(x)=F(x)\textbf{1}(x\ge 0)$ for $x \in(-\infty,\infty)$ and $\textbf{1}(A)$ is the indicator function of the set $A$.

 The class $\cal S$ contains a
lot of important distributions such as the lognormal distribution
and heavy-tailed Weibull distributions, as well as Pareto
distribution and Benktander Types I and II distributions etc.

Note that if a distribution $F$ supported on $[0,+\infty)$ or $(-\infty,+\infty)$ is subexponential, then it is long-tailed,
 denoted by $F\in\cal L$, in the sense that the relation
$$\overline{F}(x+y)\sim\overline{F}(x)$$
holds for any $y\in(-\infty, \infty)$.

The long-tailed distribution class has some important properties. For example,
it is well known that if $F\in \cal L$, then the function class
\begin{eqnarray*} {\mathcal H}(F)=:\left\{h ~{\rm{on}}~ [0,\infty):
h(x) \uparrow \infty,~{\frac{h(x)}{x} \downarrow0 }\mbox{~and~}
 \overline{F}(x-h(x))\sim \overline{F}(x)\right\}
\end{eqnarray*} is not empty, see for instance, Cline and Samorodnitsky (1994).
 Moreover, it is clear that
$h=\min\{h_1,h_2\}\in {\mathcal H}(F)$ if $h_i\in {\mathcal H}(F)$, $i=1,~2$. Additionally, if $h\in {\mathcal H}(F)$, then $ch\in {\mathcal H}(F)$ for any constant $c>0$.

For more detailed properties of the classes
$\cal S$ and $\cal L$ and their applications on finance and insurance, readers are referred to Embrechts et al. (1997) and Foss et al. (2011) etc.

Another important heavy-tailed distribution is the
dominatedly varying tailed distribution: an unbounded distribution $F$  supported on
$(-\infty,\infty)$ or $[0,\infty)$ is said to be
dominatedly varying tailed, denoted by
$F\in\cal{D}$, if the relation
$$
\overline{F}(tx)=O(\overline{F}(x))
$$
holds for some (or, equivalently, for any) $0<t<1$.

It is well known that $\mathcal L\cap\mathcal D\subset\mathcal
S$.

\subsection{Some assumptions}\

Before describing the main results of this paper, we will give some assumptions among random variables, which were introduced by Foss and Richards (2010) with a slight modification.

  Let $X_i$, $i=1,2,\cdots$ be random variables  with distributions $F_i$, $i=1,2,\cdots$. Let $F$ be a reference distribution with a subexponential tail supported on $[0,\infty)$
  and let $\mathscr{G}\subset\mathscr{F}$ be a $\sigma-$algebra.

\begin{assd}\label{asd1} r.v.s $X_1,X_2,\cdots$ are conditionally independent given $\mathscr{G}$. That is, for any collection of indices \{$i_1,\cdots,i_r$\}, and any collection of Borel sets $ A_{i_1},\cdots, A_{i_r}$, it holds that
 $$\mathbb{P}(X_{i_1}\in A_{i_1},\cdots,X_{i_r}\in A_{i_r}|\mathscr{G})=\mathbb{P}(X_{i_1}\in A_{i_1}|\mathscr{G})\cdots\mathbb{P}(X_{i_r}\in A_{i_r}|\mathscr{G}).$$\end{assd}

\begin{assd}\label{asd2} \textup{(i)} For each $i\geq 1$, $\overline F_i(x)\sim \alpha_i\overline F(x)$ for some constant $\alpha_i>0$;

 \textup{(ii)} There exists a constant $C>0$ such that $\overline F_i(x)\leq C\overline F(x)$ for all $x>0$ and $i\geq 1$.
\end{assd}

\begin{assd}\label{asd3}  For each $i\geq 1$, there exists a nondecreasing function $r(x)$ and an increasing collection of sets $B_i(x)\in\mathscr{G}$, with
$B_i(x)\to\Omega$ as $x\to\infty$,  such that the following inequality holds almost surely:
\begin{equation}\label{meq0}\mathbb{P}(X_i>x|\mathscr{G})\textbf{1}(B_i(x))\leq r(x)\overline F(x)\textbf{1}(B_i(x)),\end{equation}
and there is a function $h\in {\mathcal H}(F)$ such that

\textup{(i)} $\mathbb{P}(\overline B_i(h(x)))=o(\overline F(x))$ uniformly in $i$,

\textup{(ii)} $r(x)\overline F(h(x))=o(1)$,

\textup{(iii)} $r(x)\int_{h(x)}^{x-h(x)}\overline F(x-y)F(dy)=o(\overline F(x))$.

\end{assd}
\begin{rem}Clearly, if r.v.s $X_i,~i=1,2,\cdots$ are independent and  Assumption \ref{asd2} holds,  then Assumptions \ref{asd1} and \ref{asd3} are satisfied for $r(x)\equiv C $, $B_i(x)\equiv\Omega$ and for all $\sigma-$algebra $\mathscr{G}\subset\mathscr{F}$ and all $h\in {\mathcal H}(F)$.  \end{rem}

For further discussions on Assumptions \ref{asd1}-\ref{asd3}, readers are referred to Remarks 2.1-2.5 and Examples 1-5 in Foss and Richards(2010).

\subsection{Main results}\

In this subsection, we will present the main results of this paper. The proofs of theorems are arranged in the next section.

 First, inspired by Theorem 3.1 in Tang and Yuan (2014), we will investigate the tail behavior of the randomly weighted sums in which the primary r.v.s $X_1, \cdots , X_n,\cdots$ are conditionally independent and follow subexponential distributions.
\begin{thm}\label{th0}
Let $X_i$, $i=1,2,\cdots$ be conditionally independent nonnegative r.v.s with distributions $F_i,~i=1,2,\cdots$, which satisfy Assumptions \ref{asd1}, \ref{asd2}(i) and \ref{asd3} for a $\sigma$-algebra $\mathscr{G}\subset\mathscr{F}$ and a reference distribution $F\in \cal S$.
Suppose that r.v.s $\theta_i$, $i=1,2, \cdots$ are 
 positive and bounded above, namely there is a positive constant $b$ such that \begin{eqnarray}\label{tbound}
    P(0<\theta_i\le b)=1,\   i=1,2,\cdots,
\end{eqnarray}
which is independent of r.v.s $X_i,~i=1,2,\cdots$.
    Then for any $n=1,2,\cdots$, it holds that
\begin{eqnarray}\label{rws}
\mathbb{P}\lo(\sum_{i=1}^{n}\theta_i X_i>x\ro)\sim\mathbb{P}\lo(\max_{1\leq i\leq n}\theta_i X_i>x\ro)\sim\sum_{i=1}^{n}\mathbb{P}(\theta_i X_i>x).
\end{eqnarray}
\end{thm}

Next, we will give a result of the finite-time ruin probability,
 which extends Theorem 3.1 in Tang (2005) from the independent claim sizes to the conditionally independent one, and from the homogeneous Poisson claim number to nonhomogeneous one.
\begin{thm}\label{th1}Consider the compound nonhomogeneous Poisson model with a constant interest force introduced in section 1, in which the claim sizes $X_k, k=1,2,\cdots,$ are conditionally independent r.v.s with a common distribution $F\in\mathcal{S}$ which satisfy Assumptions \ref{asd1} and \ref{asd3} for a $\sigma$-algebra $\mathscr{G}\subset\mathscr{F}$, while the arrival times $\sigma_k, k=1,2,\cdots$, constitute a nonhomogeneous Poisson process $\{N(t): t\geq 0\}$ with an intensity function $\lambda(t)$.
 Let $\{C(t): t\geq 0\}$ be a nondecreasing and right-continuous stochastic process, denoting the total premium accumulated up to time $t$ and let $r>0$ be the constant interest force.
Assume that $\{X_k,k\geq 1\}, \{\sigma_k,k\geq 1\}$ and $\{C(t),t\geq 0\}$ are mutually independent.

 Then for any fixed $T>0$, it holds that
\begin{eqnarray}\label{50}
\psi_r(x,T)\sim\int_{0}^{T}\overline F(xe^{ru})\lambda(u)du.
\end{eqnarray}
\end{thm}

\section{Proofs of Theorems}
\setcounter{thm}{0}\setcounter{Corol}{0}\setcounter{lemma}{0}\setcounter{pron}{0}\setcounter{equation}{0}
\setcounter{remark}{0}\setcounter{exam}{0}\setcounter{property}{0}\setcounter{defin}{0}\

 Before proving Theorems \ref{th0} and \ref{th1}, we prepare some lemmas on randomly weighted sums.

Inspired by Proposition 5.1 of Tang and Tsitsiashvili (2003), the first lemma studies the uniformly asymptotic behavior of weighted sums
of conditionally independent subexponential increments.
\begin{lemma}\label{lem1}
Let $X_i$, $i=1,2,\cdots$ be conditionally independent nonnegative r.v.s with distributions $F_i,~i=1,2,\cdots$, which satisfy Assumptions \ref{asd1}, \ref{asd2}(i) and \ref{asd3} for a $\sigma$-algebra $\mathscr{G}\subset\mathscr{F}$ and a reference distribution $F\in \cal S$. Then for any fixed $0<a<b<\infty$ and any positive integer $n$, the relation
\begin{eqnarray}\label{1}
\mathbb{P}\lo(\sum_{i=1}^{n}c_i X_i>x\ro)\sim\sum_{k=1}^{n}\mathbb{P}\lo(c_i X_i>x\ro)
\end{eqnarray}
holds uniformly for all $(c_1,\cdots,c_n)\in[a,b]^n$, that is
\begin{eqnarray}
\limsup\sup\limits_{(c_1,\cdots,c_n)\in
{[a,b]}^n}\left|\dfrac{P\left(\sum_{i=1}^n c_iX_i>x\right)}{
\sum_{i=1}^{n}P\left(c_iX_i>x\right)}- 1\right |=0.\label{uni01}
\end{eqnarray}
\end{lemma}

\begin{proof}We will prove (\ref{uni01}) by mathematical induction. When $n=1$, the relation (\ref{uni01}) holds naturally. Then we assume that (\ref{uni01}) holds for $n=m$, where $m$ is a positive integer. We are aim to prove that
(\ref{uni01}) holds for $n=m+1$.

Fixed $\varepsilon>0$. From the induction hypothesis, there is a positive constant $x_0$ such that
\begin{eqnarray}
(1-\varepsilon)\sum_{i=1}^{m}P(c_iX_i>x)\leq P\left(\sum_{i=1}^m c_iX_i>x\right)\leq
(1+\varepsilon)\sum_{i=1}^{m}P(c_iX_i>x).\label{uni03}
\end{eqnarray}
holds for all $(c_1,\cdots,c_m)\in
{[a,b]}^m$ and $x>x_0$. Now we fix weighted numbers $(c_1,\cdots,c_m,c_{m+1})\in
{[a,b]}^{m+1}$.
For notational convenience we write $S_m^{c}=\sum_{k=1}^{m}c_k X_k$ and write $M=\max\{c_1,\cdots,c_{m+1}\}$.

First, we discuss the upper bound. By conventional methods, we use a decomposition as follows: 
\begin{eqnarray}
\mathbb{P}\lo(\sum_{i=1}^{m+1} c_iX_i>x\ro)
&\leq&\mathbb{P}\lo(c_{m+1} X_{m+1}>x-{M}h\left(\frac{x}{{M}}\right)\ro)+\mathbb{P}\lo(S_m^{c}> x-{M}h\left(\frac{x}{{M}}\right)\ro)\nonumber\\
& &+\mathbb{P}\lo({M}h\left(\frac{x}{{M}}\right)<S_m^{c}\leq x-{M}h\left(\frac{x}{{M}}\right),\sum_{i=1}^{m+1} c_iX_i>x\ro)\nonumber\\
&=:&L_1(x)+L_2(x)+L_3(x). \label{fj01}
\end{eqnarray}
Firstly, we estimate $L_1(x)$. By Assumption \ref{asd2}(i) and $h\in\mathcal{H}(F)$, there is a positive constant $x_1\geq x_0$ such that
\begin{eqnarray}\label{uni04}\overline F_{i}(x-bh(x)/a)\leq(1+\varepsilon)\overline F_{i}(x)
\end{eqnarray}
and
\begin{eqnarray}
 \frac{\alpha_i}{2}\overline{F}(x)\leq \overline{F}_i(x)\leq 2\alpha_i\overline{F}(x) \label{meq4}
 \end{eqnarray}
 hold for all $x>x_1$ and $i=0,1,2,\cdots, m+1$, where $F_0\equiv F$ and $\alpha_0\equiv 1$  by convention. Hence, it follows from (\ref{fj01}) and (\ref{uni04}) that
 \begin{eqnarray}\label{l1}
L_1(x)
&\leq&\mathbb{P}\lo( X_{m+1}>\frac{x}{c_{m+1}}-\frac{b}{a}h\lo(\frac{x}{c_{m+1}}\ro)\ro)\nonumber\\
&\leq&(1+\varepsilon)\mathbb{P}\lo(c_{m+1} X_{m+1}>x\ro)
\end{eqnarray}
holds for all $x>(b\vee 1)x_1.$

Next, we estimate $L_2(x)$.
Since $h(x) \uparrow\infty$ and $\frac{h(x)}{x}\downarrow 0$, there is a positive number $x_2$ such that
$$x-bh\left(\frac{x}{a}\right)>x_0 ~\text{ and }~  (a\wedge 1)h\left(\frac{x}{b}\right)>x_1$$
hold for all $x>x_2$, which yields that
\begin{equation}\label{uni07}x-Mh\left(\frac{x}{M}\right)>x_0\end{equation}
and\begin{equation}\label{uni08}(M\wedge 1)h\left(\frac{x}{M}\right)>x_1\geq x_0\end{equation}
hold for all $x>x_2$. Hence, it follows from (\ref{uni03}), (\ref{uni04}) and (\ref{uni07}) that
\begin{eqnarray}\label{l2}
L_2(x)&\leq&(1+\varepsilon)\sum_{i=1}^{m}\mathbb{P}\lo(c_i X_i> x-{M}h\left(\frac{x}{{M}}\right)\ro)\nonumber\\
&\leq&(1+\varepsilon)\sum_{i=1}^{m}\mathbb{P}\lo(X_i> \frac{x}{c_i}-\frac{b}{a}h\lo(\frac{x}{c_i}\ro)\ro)\nonumber\\
&\leq&(1+\varepsilon)^2\sum_{i=1}^{m}\mathbb{P}\lo(c_iX_i> x\ro)
\end{eqnarray}
holds for all $x>\max((b\vee 1)x_1, x_2)$.

 To complete the estimate on the upper bound, we now estimate $L_3$.
Let $Y$ be an r.v. with distribution $F\in\mathcal{S}$, independent of $X_i$, $1\le i \le m+1$ and $\mathscr{G}$. Using the approach followed in Proposition 2.1 of Foss and Richards (2010),
by the law of total probability, it follows from the conditional independence between $X_{m+1}$ and $S_m^c$ that
\begin{eqnarray}\label{5}
&&L_3(x)
=\mathbb{E}\lo(\mathbb{P}\lo(\left.{M}h\left(\frac{x}{{M}}\right)<S_m^{c}\leq x-{M}h\left(\frac{x}{{M}}\right),\sum_{i=1}^{m+1} c_iX_i>x\right|\mathscr{G}\ro)\ro)\nonumber\\
&\leq&\mathbb{E}\lo(\int_{{M}h\left(\frac{x}{{M}}\right)}^{x-{M}h\left(\frac{x}{{M}}\right)}\mathbb{P}\Big(\left.X_{m+1}>\frac{x-y}{{M}}\right|\mathscr{G}\Big)\mathbb{P}\lo(S_{m}^{c}\in dy|\mathscr{G}\ro)
\ro)\nonumber\\
&=&\mathbb{E}\Bigg(\int_{{M}h\left(\frac{x}{{M}}\right)}^{x-{M}h\left(\frac{x}{{M}}\right)}\mathbb{P}\Big(\left.X_{m+1}>\frac{x-y}{{M}}\right|\mathscr{G}\Big)\mathbb{P}\lo(S_{m}^{c}\in dy|\mathscr{G}\ro)
 \cdot\Big(\textbf{1}\Big(B_{m+1}\Big(\frac{x-y}{{M}}\Big)\Big)+\textbf{1}\Big(\overline B_{m+1}\Big(\frac{x-y}{{M}}\Big)\Big)\Big)\Bigg)\nonumber\\
&\leq&r\left(\frac{x}{{M}}\right)\mathbb{E}\lo(\int_{{M}h\left(\frac{x}{{M}}\right)}^{x-{M}h\left(\frac{x}{{M}}\right)}\mathbb{P}\Big(Y>\frac{x-y}{M}\Big|\mathscr{G}\Big)\mathbb{P}\lo(S_{m}^{c}\in dy|\mathscr{G}\ro)
\ro)+\mathbb{E}\Big(\textbf{1}\Big(\overline B_{m+1}\Big(h\left(\frac{x}{{M}}\right)\Big)\Big)\Big)\nonumber\\
&=&r\left(\frac{x}{{M}}\right)\mathbb{P}\Big({M}Y+S_m^c>x, {M}h\left(\frac{x}{{M}}\right)<S_m^c<x-{M}h\left(\frac{x}{{M}}\right)\Big)+\mathbb{P}\Big( \overline B_{m+1}\Big(h\left(\frac{x}{{M}}\right)\Big)\Big),\label{uni05}
\end{eqnarray}
where the last but one step comes from (\ref{meq0}). To estimate the first term in (\ref{uni05}), it follows from (\ref{uni03}), (\ref{uni04}) and (\ref{uni08})  that
\begin{eqnarray}
&&\mathbb{P}({M}Y+S_m^c>x, {M}h\left(\frac{x}{{M}}\right)<S_m^c<x-{M}h\left(\frac{x}{{M}}\right))\nonumber\\
&\leq&\mathbb{P}\Big ({M}Y>x-{M}h \Big (\frac{x}{{M}} \Big ), S_m^c>{M}h \Big (\frac{x}{{M}} \Big )\Big )+\mathbb{P}\Big ({M}Y+S_m^c>x, h \Big (\frac{x}{{M}} \Big )<Y<\frac{x}{M}-h \Big (\frac{x}{{M}} \Big )\Big )\nonumber\\
&\leq& \mathbb{P}\lo(S_{m}^{c}>{M}h\left(\frac{x}{{M}}\right)\ro)\overline F\left(\frac{x}{{M}}-h\left(\frac{x}{{M}}\right)\right)
 + \int_{{M}h\left(\frac{x}{{M}}\right)}^{x-{M}h\left(\frac{x}{{M}}\right)}\mathbb{P}\lo(S_{m}^{c}>x-y\ro)\mathbb{P}({M}Y\in dy) \nonumber\\
&\leq&(1+\varepsilon)^2 \sum_{i=1}^{m}\mathbb{P}\lo(c_i X_i>{M}h\left(\frac{x}{{M}}\right)\ro)\overline F\Big(\frac{x}{{M}}\Big)\nonumber\\
& &+ (1+\varepsilon)\sum_{i=1}^{m}\int_{{M}h\left(\frac{x}{{M}}\right)}^{x-{M}h\left(\frac{x}{{M}}\right)}\mathbb{P}\lo(c_i X_i>x-y\ro)\mathbb{P}({M}Y\in dy)\label{uni09}
\end{eqnarray}
holds for all $x>\max((b\vee 1)x_1,x_2)$. Recall that $Y$ is independent of $X_i$, $i=1,2,\cdots, m+1$.
Noting that $\frac{x}{{M}}-y>h\left(\frac{x}{{M}}\right)$ for all $h\left(\frac{x}{{M}}\right)\leq y\leq\frac{x}{{M}}-h\left(\frac{x}{{M}}\right)$,
it follows from (\ref{meq4}) and (\ref{uni08})  that
\begin{eqnarray}&&\sum_{i=1}^{m}\int_{{M}h\left(\frac{x}{{M}}\right)}^{x-{M}h\left(\frac{x}{{M}}\right)}\mathbb{P}\lo(c_i X_i>x-y\ro)\mathbb{P}({M}Y\in dy)\nonumber\\
&\leq& 
 \sum_{i=1}^{m}\int_{h\left(\frac{x}{{M}}\right)}^{\frac{x}{{M}}-h\left(\frac{x}{{M}}\right) }\mathbb{P}\lo(X_i>\frac{x}{{M}}-y\ro)F(dy) \nonumber\\
&\leq&m\rho  \int_{h\left(\frac{x}{{M}}\right)}^{\frac{x}{{M}}-h\left(\frac{x}{{M}}\right)}\overline F\lo(\frac{x}{{M}}-y\ro)F(dy) \label{uni11}
\end{eqnarray}
holds for all $x>x_2$, where $\rho=\rho_{m+1}=2\max(\alpha_1,\cdots,\alpha_{m+1}).$
Substituting (\ref{uni09}) and (\ref{uni11}) into (\ref{uni05}), we have that
\begin{eqnarray}L_3(x)&\leq&m\rho(1+\varepsilon)^2r\left(\frac{x}{{M}}\right)\overline F\left(h\left(\frac{x}{{M}}\right)\right)\overline F\left(\frac{x}{{M}}\right)\nonumber\\
&  &+m\rho(1+\varepsilon)r\left(\frac{x}{{M}}\right)\int_{h\left(\frac{x}{{M}}\right)}^{\frac{x}{{M}}-h\left(\frac{x}{{M}}\right)}\overline F\lo(\frac{x}{{M}}-y\ro)F(dy)+\mathbb{P}\left( \overline B_{m+1}\left(h\left(\frac{x}{{M}}\right)\right)\right)\label{uni12}
\end{eqnarray}
holds for all $x>\max((b\vee 1)x_1,x_2)$.
From Assumption \ref{asd3}, there exists a positive constant $x_3$ such that
\begin{eqnarray}
 &&\mathbb{P}( \overline B_{m+1}(h(x)))<\varepsilon\overline{F}(x);\label{meq1}\\
 &&r(x)\overline{F}(h(x))<\varepsilon;\label{meq2}
 \end{eqnarray}
and
\begin{equation*}
r(x )\int_{h(x )}^{x -h(x )}\overline F\lo(x-y\ro)F(dy)<\varepsilon \overline{F}(x)\label{meq3}
\end{equation*}
holds for all $x>x_3$.
  Hence, 
  combining with (\ref{uni12}), we have that
\begin{eqnarray*}
L_3(x)\leq( m\rho\varepsilon(1+\varepsilon)^2+m\rho\varepsilon(1+\varepsilon)+\varepsilon)\overline{F}(\frac{x}{{M}})
\end{eqnarray*}
holds for all $x>\max((b\vee 1)x_1,x_2,(b\vee 1)x_3)$.
Now we denote $\kappa=\min(\alpha_i,~1\leq i\leq m+1)$. Combining with (\ref{meq4}), we have that
\begin{eqnarray}\label{l3}
L_3(x)\leq \frac{2\varepsilon}{\kappa}( m\rho (1+\varepsilon)^2+m\rho (1+\varepsilon)+1)\sum_{i=1}^{m+1}P(c_iX_i>x).
\end{eqnarray}
holds for all $x>\max((b\vee 1)x_1,x_2,(b\vee 1)x_3)$. By the arbitrariness of $\varepsilon$, noting that numbers $x_i$, $i=0,1,2,3$ are independent of $c_i$, $i=1,2,\cdots,m+1$, it follows from (\ref{l1}),(\ref{l2}) and (\ref{l3}) that
\begin{eqnarray}
\limsup\sup\limits_{(c_1,\cdots,c_{m+1})\in
{[a,b]}^{m+1}}\dfrac{P(\sum_{i=1}^{m+1} c_iX_i>x)}{
\sum_{i=1}^{m+1}P(c_iX_i>x)}\leq 1.\label{uni02}
\end{eqnarray}

Now we will estimate the lower bound. Since all $c_i$ and $X_i$, $i=1,2,\cdots$ are nonnegative, we can simply use the following decomposition:
\begin{eqnarray}\label{uni21}
\mathbb{P}\lo(\sum_{i=1}^{m+1} c_iX_i>x\ro)\geq\mathbb{P}\lo(S_m^{c}> x\ro)+\mathbb{P}\lo(c_{m+1} X_{m+1}> x\ro)-\mathbb{P}\lo(S_{m}^{c}> x,c_{m+1} X_{m+1}> x\ro)\end{eqnarray}
From (\ref{uni03}), it is clear that
\begin{eqnarray}\label{uni22}
\mathbb{P}\lo(S_m^{c}> x\ro)+\mathbb{P}\lo(c_{m+1} X_{m+1}> x\ro)&\geq&(1-\varepsilon)\sum_{k=1}^{m+1}\mathbb{P}\lo(c_k X_k> x\ro)
\end{eqnarray}
holds for all $x>x_0$. Therefore, we only need to give an estimation of the third term at the right hand of (\ref{uni21}).
Using the law of total probability again, it follows from (\ref{meq0}) that
 \begin{eqnarray*}\label{7}
&&\mathbb{P}\lo(S_m^{c}> x,c_{m+1} X_{m+1}> x\ro)\nonumber\\
&\leq&\mathbb{E}\lo(\mathbb{P}\lo(S_m^{c}> x,X_{m+1}> \frac{x}{{M}}\Big|\mathscr{G}\ro)\left(\textbf{1}\left(B_{m+1}\left(\frac{x}{{M}}\right)\right)+\textbf{1}\left(\overline B_{m+1}\left(\frac{x}{{M}}\right)\right)\right)\ro)\nonumber\\
&\leq&\mathbb{E}\lo(\mathbb{P}\lo(S_m^{c}> x|\mathscr{G}\ro)\mathbb{P}\lo(X_{m+1}> \frac{x}{{M}}\Big|\mathscr{G}\ro)\textbf{1}\left(B_{m+1}\left(\frac{x}{{M}}\right)\right)\ro)+\mathbb{E}\left(\textbf{1}\left(\overline B_{m+1}\left(\frac{x}{{M}}\right)\right)\right)\nonumber\\
&\leq&r\left(\frac{x}{{M}}\right)\overline F\left(\frac{x}{{M}}\right)\mathbb{P}\lo(S_m^{c}> x\ro)+\mathbb{P}\left(\overline B_{m+1}\left(\frac{x}{{M}}\right)\right),
\end{eqnarray*}
which implies from (\ref{uni03}), (\ref{meq1}),(\ref{meq2}) and (\ref{meq4}) that
\begin{eqnarray}
&&\mathbb{P}\lo(S_m^{c}> x,c_{m+1} X_{m+1}> x\ro)\nonumber\\
&\leq&(1+\varepsilon)r\left(\frac{x}{{M}}\right)\overline F\left(\frac{x}{{M}}\right)\sum_{k=1}^{m}\mathbb{P}\lo(c_k X_k> x\ro)+\varepsilon\overline F\left(\frac{x}{{M}}\right)\nonumber\\
&\leq& \varepsilon\left (1+\varepsilon+\frac{2}{\kappa}\right)\sum_{k=1}^{m+1}\mathbb{P}\lo(c_k X_k> x\ro)
.\label{uni25}
\end{eqnarray}
holds for all $x>\max(x_0,(b\vee 1)x_1,(b\vee 1)x_3)$.
By the arbitrariness of $\varepsilon$, noting that numbers $x_i$, $i=0,1,2,3$ are independent of $c_i$, $i=1,2,\cdots,m+1$ again, it follows from (\ref{uni21}),(\ref{uni22}) and (\ref{uni25}) that
\begin{eqnarray}
\liminf\inf\limits_{(c_1,\cdots,c_{m+1})\in
{[a,b]}^{m+1}}\frac{\mathbb{P}(\sum_{i=1}^{m+1} c_iX_i>x)}{
\sum_{i=1}^{m+1}\mathbb{P}(c_iX_i>x)}\geq 1.\label{uni20}
\end{eqnarray}

Hence, by (\ref{uni02}) and (\ref{uni20}), we obtain that (\ref{uni01}) holds for $n=m+1$. 
\end{proof}

In the next lemma, we establish a 
 result on uniform asymptotic behavior of weighted maximum of subexponential increments.
\begin{lemma}\label{lem2}
Let $X_i$, $i=1,2,\cdots$ be conditionally independent real valued r.v.s with distributions $F_i,~i=1,2,\cdots$, which satisfy Assumptions \ref{asd1}, \ref{asd2}(i) and \ref{asd3} for a $\sigma$-algebra $\mathscr{G}\subset\mathscr{F}$ and a reference distribution $F\in \cal S$. Then for any fixed $0<b<\infty$ and any positive integer $n$, the relation
\begin{eqnarray}\label{max1}
\mathbb{P}\lo(\max_{1\leq i\leq n}c_i X_i>x\ro)\sim\sum_{i=1}^{n}\mathbb{P}\lo(c_i X_i>x\ro)
\end{eqnarray}
holds uniformly for all $(c_1,\cdots,c_n)\in[0,b]^n$. When all the weights $c_i$ are equal to $0$, the ratio of the left and right hands of (\ref{max1}) is simply understood as 1.
 \end{lemma}

\begin{proof}It is trivial that for all $({c_1}, {c_2}, \ldots , {c_{n}})
\in [0,b]^n$, it holds that
$$\mathbb{P}\left(\max_{1\le i\le n} c_iX_i>x\right)\le
  \sum_{i=1}^n \mathbb{P}\left(
c_iX_i>x\right).$$ Hence, it remains to prove that
 \be\label{mx}
\liminf\inf\limits_{(c_1,\cdots,c_{n})\in
{[a,b]}^{n}}\frac{\mathbb{P}\left(\max_{1\le i\le n} c_iX_i>x\right)}{
  \sum_{i=1}^n \mathbb{P}\left(
c_iX_i>x\right)}\geq 1.\ee

Let $M=\max_{1\leq i\leq n}c_i>0$.
For any $1\leq i<j\leq n$, it follows from Assumption \ref{asd1} and (\ref{meq0}) that
\begin{eqnarray}&&\mathbb{P}\lo (c_i X_i>x, c_j X_j>x\ro )\nonumber\\
&\leq& \mathbb{E}\lo (\mathbb{P}\lo (X_i>\frac{x}{M},X_j>\frac{x}{M}|\mathscr{G}\ro )\Big (\textbf{1}\Big (B_{j}\Big (\frac{x}{M}\Big )\Big )+\textbf{1}\Big (\overline B_{j}\Big (\frac{x}{M}\Big )\Big )\Big )\ro )\nonumber\\
&\leq&\mathbb{E}\lo (\mathbb{P}\lo (X_i>\frac{x}{M}|\mathscr{G}\ro )\mathbb{P}\lo (X_j>\frac{x}{M}|\mathscr{G}\ro )\textbf{1}\Big (B_{j}\Big (\frac{x}{M}\Big )\Big )\ro )+\mathbb{E}\Big (\textbf{1}\Big (\overline B_{j}\Big (\frac{x}{M}\Big )\Big )\Big )\nonumber\\
&\leq&r\Big (\frac{x}{M}\Big )\overline F\Big (\frac{x}{M}\Big )\mathbb{P}\lo (X_i>\frac{x}{M}\ro )+\mathbb{P}\Big (\overline B_{j}\Big (\frac{x}{M}\Big )\Big ).\label{max2}
\end{eqnarray}
Hence, for any fixed $\varepsilon\in (0,\frac{1}{2n^2})$, from Assumptions \ref{asd2}(i) and \ref{asd3}(i)(ii), there exists a constant $x_5>0$, which is independent of the weights $c_i$, $i=1,2,\cdots,n$, such that
$$r (x )\overline F(x )<\varepsilon~ \textup{ and }~\mathbb{P}(\overline B_{j}(x ))<\varepsilon\overline F_j(x )$$
holds for all $x>x_5$ and $j=1,2,\cdots,n.$
Noting that $$\mathbb{P}\lo(X_i>\frac{x}{M}\ro)\leq \sum_{k=1}^n \mathbb{P}\lo(c_kX_k>x\ro),$$it follows from (\ref{max2}) that
\begin{eqnarray}\label{62}
\mathbb{P}\lo(\max_{1\leq k \leq n}c_i X_i>x\ro)
&\geq&\sum_{i=1}^{n}\mathbb{P}\lo(c_i X_i>x\ro)-\sum_{1\leq i<j \leq n}\mathbb{P}\lo(c_i X_i>x, c_j X_j>x\ro)\nonumber\\
&\geq&
  (1-2n(n-1)\varepsilon)\sum_{i=1}^n \mathbb{P}\left(c_iX_i>x\right)
\end{eqnarray} holds for all $({c_1}, {c_2}, \ldots ,
{c_{n}}) \in [0,b]^n$ and $x>bx_5$, which yields that (\ref{mx}) holds by the arbitrariness of $\varepsilon$ and ends the proof of Lemma \ref{lem2}.
\end{proof}

The following lemma is a particular case of Lemma 2.1 in Foss and Richards (2010).
\begin{lemma}\label{lem4}
Let  $X_i$, $i=1,2,\cdots$ be conditionally independent r.v.s with a common distribution $F\in\mathcal{S}$ which satisfy Assumptions \ref{asd1} and \ref{asd3} for a $\sigma$-algebra $\mathscr{G}\subset\mathscr{F}$. 
 Then for any $\varepsilon>0$, there exist constants $V(\varepsilon)>0$ and $x_0=x_0(\varepsilon)$ such that, for all $x>x_0$ and $n\ge1$,
\begin{eqnarray*}
\mathbb{P}\left(\sum_{k=1}^nX_k>x\right)\leq V(\varepsilon)(1+\varepsilon)^n \overline F(x).
\end{eqnarray*}
\end{lemma}
The last lemma is well known and can be found in Ross (1983) etc.
\begin{lemma}\label{lem5}
Let $\{N(t):~t\ge0\}$ be a nonhomogeneous Poisson process with an intensity function $\lambda(t)$ and arrival times $\{\sigma_k,k\ge1\}$.
For any fixed $t>0$ and $n=1,2,\cdots$, given $N(t)=n$, the random vector $(\sigma_1,\sigma_2,\cdots,\sigma_n)$ is equal in distribution to the random vector $(Y_{(1:n)},Y_{(2:n)},\cdots,Y_{(n:n)})$, where $Y_{(1:n)},Y_{(2:n)},\cdots,Y_{(n:n)}$ are the order statistics of independent and identically distributed r.v.s $Y_1,\cdots,Y_n$ with a common density function \begin{equation}\label{nonhom1}g_t(s)=\frac{\lambda(s)}{\Lambda(t)},~0\leq s\leq t,\end{equation}
where $\Lambda(t)=\int_{0}^{t}\lambda(u)du,~t\geq 0.$
\end{lemma}

Now we are standing in a position to prove Theorem \ref{th0}.

\begin{proof}[Proof of Theorem \ref{th0}.]   By copying the proof of Theorem 2.2 of Cheng and Cheng (2017),  (\ref{rws}) follows from Lemmas \ref{lem1} and \ref{lem2} immediately.

 \end{proof}

At the end of this paper, we give the proof of Theorem \ref{th1}.

\begin{proof}[Proof of Theorem \ref{th1}.]

We will prove the theorem along the technical line of the proof of Theorem 3.1 in Tang (2005). Obviously, it follows from (\ref{ruin0}) and (\ref{ruin1}) that
\begin{eqnarray*}
\psi_r(x,T)&=&\mathbb{P}\lo( e^{-rt}S_r(t)<0 \mbox{~for some~} t\in (0,T] |S_r(0)=x\ro)\nonumber\\
&=&\mathbb{P}\lo(   \sum_{k=1}^{N(t)}X_ke^{-r\sigma_k}>x+\int_{0}^{t}e^{-rs}C(ds) \mbox{~for some~} t\in (0,T] \Bigg|S_r(0)=x\ro),\label{pru1}
\end{eqnarray*}
hence, we have
\begin{eqnarray}\label{pru2}
\mathbb{P}\lo(\sum_{k=1}^{N(T)}X_ke^{-r\sigma_k}>x+\int_{0}^{T}e^{-rs}C(ds)\ro)\leq \psi_r(x,T)\leq\mathbb{P}\lo(\sum_{k=1}^{N(T)}X_ke^{-r\sigma_k}>x\ro).
\end{eqnarray}
First, we estimate the upper bound $\mathbb{P}\lo(\sum_{k=1}^{N(T)}X_ke^{-r\sigma_k}>x\ro)$.
Let $Y_1,Y_2,\cdots$ be independent and identically distributed r.v.s with a common density function $g_T(\cdot)$ defined in (\ref{nonhom1}), which is independent of r.v.s $X_1,X_2,\cdots$. The corresponding distribution function is denoted by $$G_T(x)=\int_0^xg_T(u)du,~ ~x\geq 0.$$
By Lemma \ref{lem5}, we have
\begin{eqnarray}\label{51}
\mathbb{P}\lo(\sum_{k=1}^{N(T)}X_ke^{-r\sigma_k}>x\ro)&=&
\sum_{n=1}^{\infty}\mathbb{P}\lo(\sum_{k=1}^{n}X_ke^{-r\sigma_k}>x|N(T)=n\ro)\mathbb{P}(N(T)=n)\nonumber\\
&=&\sum_{n=1}^{\infty}\mathbb{P}\lo(\sum_{k=1}^{n}X_ke^{-rY_{(k:n)}}>x\ro)\mathbb{P}(N(T)=n),
\end{eqnarray}
where  $(Y_{(1:n)},Y_{(2:n)},\cdots,Y_{(n:n)})$ are the order statistics of $Y_1,\cdots,Y_n$ for any $n\geq 1$.
For any fixed $\varepsilon_0>0$ and $n\geq1$, noting that
  $$e^{-rT}\leq e^{-rY_{(k:n)}}\leq e^{-rY_{(1:n)}},~k=1,\cdots,n,$$ it follows from Lemma \ref{lem4} that 
\begin{eqnarray*}\label{52}
 &&\mathbb{P}\lo(\sum_{k=1}^{n}X_ke^{-rY_{(k:n)}}>x\ro) \leq\mathbb{P}\lo(\sum_{k=1}^{n}X_k>xe^{rY_{(1:n)}}\ro) \nonumber\\
&=& \int_{0}^{T}\mathbb{P}\lo(\sum_{k=1}^{n}X_k>xe^{ru}\ro)n(1- G_T(u))^{n-1}g_T(u)du \nonumber\\
&\leq& \frac{V(\varepsilon_0)}{\Lambda(T)} n(1+\varepsilon_0)^n
\lo(\int_{0}^{T}\overline F(xe^{ru})\lambda(u)du\ro).
\end{eqnarray*}
Hence, combining with 
\begin{eqnarray*}\label{53}
\sum_{n=1}^n\frac{V(\varepsilon_0)}{\Lambda(T)} n(1+\varepsilon_0)^n
\mathbb{P}(N(T)=n)=\frac{V(\varepsilon_0)}{\Lambda(T)}EN(T)(1+\varepsilon_0)^{N(T)}
 <\infty.
\end{eqnarray*}
by the dominated convergence theorem, it follows from (\ref{51}) and Theorem \ref{th0} that
\begin{eqnarray}\label{54}
 &&\mathbb{P}\lo(\sum_{k=1}^{N(T)}X_ke^{-r\sigma_k}>x\ro)\sim\sum_{n=1}^{\infty}\sum_{k=1}^{n}\mathbb{P}(
X_ke^{-rY_{(k:n)}}>x)\mathbb{P}(N(T)=n)\nonumber\\
&=&\sum_{n=1}^{\infty}\sum_{k=1}^{n}\mathbb{P}(
X_1e^{-rY_{(k:n)}}>x)\mathbb{P}(N(T)=n)\nonumber\\
&=&\sum_{n=1}^{\infty}\sum_{k=1}^{n}\lo(\int_{0}^{T}\mathbb{P}(
X_1e^{-ru}>x)\frac{n!}{(k-1)!(n-k)!}G_T^{k-1}(u)(1-G_T(u))^{n-k}g_T(u)du\ro)\mathbb{P}(N(T)=n)\nonumber\\
&=&\sum_{n=1}^{\infty}n\mathbb{P}(N(T)=n)\lo(\int_{0}^{T}\overline F(xe^{ru})\sum_{k=1}^{n}\frac{(n-1)!}{(k-1)!(n-k)!}G_T^{k-1}(u)(1-G_T(u))^{n-k}g_T(u)du\ro)\nonumber\\
&=&\frac{1}{\Lambda(T)}\sum_{n=1}^{\infty}n\mathbb{P}(N(T)=n)\lo(\int_{0}^{T}\overline F(xe^{ru})\lambda(u)du\ro) \nonumber\\
&=&\int_{0}^{T}\overline F(
xe^{ru})\lambda(u)du. \label{55}
\end{eqnarray}
Now we estimate the lower bound: Note that
$$\int_{0}^{T}\overline F((x+y)e^{ru})\lambda(u)du\sim \int_{0}^{T}\overline F(xe^{ru})\lambda(u)du$$
holds for all $y>0$. It follows from (\ref{55}) that the r.v.
 $\sum_{k=1}^{N(T)}X_ke^{-r\sigma_k}$ follows a long-tailed distribution. Hence, using the dominated convergence theorem, it follows from the independence between $\int_{0}^{T}e^{-rs}C(ds)$ and $\sum_{k=1}^{N(T)}X_ke^{-r\sigma_k}$ that
\begin{eqnarray}
&&\lim\frac{\mathbb{P}\lo(\sum_{k=1}^{N(T)}X_ke^{-r\sigma_k}>x+\int_{0}^{T}e^{-rs}C(ds)\ro)}
{\mathbb{P}\lo(\sum_{k=1}^{N(T)}X_ke^{-r\sigma_k}>x\ro)}\nonumber\\
&=&\int_{0}^{\infty}\lim\frac{\mathbb{P}\lo(\sum_{k=1}^{N(T)}X_ke^{-r\sigma_k}>x+y\ro)}
{\mathbb{P}\lo(\sum_{k=1}^{N(T)}X_ke^{-r\sigma_k}>x\ro)}\mathbb{P}\lo(\int_{0}^{T}e^{-rs}C(ds)\in dy\ro)=1,\label{pru5}
\end{eqnarray}
Hence, (\ref{50}) follows from (\ref{pru2}), (\ref{55}) and (\ref{pru5}). 
\end{proof}

\end{CJK*}
\end{document}